\documentclass[a4paper,11pt]{article}
\usepackage{amsfonts,amssymb,amsthm,cite,amsmath,amstext}
\usepackage[colorlinks,linkcolor=red,citecolor=blue]{hyperref}
\numberwithin{equation}{section}

\title{\textbf{Weak transcendental holomorphic Morse inequalities on compact K\"{a}hler manifolds }}
\author{\textsc{Jian Xiao}}
\date{}
\begin{document}
\maketitle

\theoremstyle{definition}
\newtheorem*{pf}{Proof}
\newtheorem{theorem}{Theorem}[section]
\newtheorem{remark}{Remark}[section]
\newtheorem{problem}{Problem}[section]
\newtheorem{conjecture}{Conjecture}[section]
\newtheorem{lemma}{Lemma}[section]
\newtheorem{corollary}{Corollary}[section]
\newtheorem{definition}{Definition}[section]
\newtheorem{proposition}{Proposition}[section]

\begin{abstract}
Transcendental holomorphic Morse inequalities aim at characterizing the positivity of transcendental cohomology classes of type $(1,1)$. In this paper, we prove a weak version of Demailly's conjecture on transcendental Morse inequalities on compact K\"{a}hler manifolds. And as a consequence, we partially improve a result of Boucksom-Demailly-Paun-Peternell.
\end{abstract}
\tableofcontents

\section{Introduction}

There are many beautiful results on holomorphic Morse inequalities for rational cohomology classes of type $(1,1)$. For rational cohomology classes of type $(1,1)$ which are first Chern classes of holomorphic $\mathbb{Q}$-line bundles, these inequalities are related to the holomorphic sections of line bundles. Demailly has applied holomorphic Morse inequalities for rational cohomology classes to reprove a stronger statement of Grauert-Riemenschneider conjecture (ref. \cite{Dem85, Siu84}). Recently, these inequalities are also applied to the Green-Griffiths-Lang conjecture (ref. \cite{Dem11}).

However, if the cohomology class is not rational, which we also call transcendental class, we do not have holomorphic sections for these cohomology classes, it is hard to prove the associated holomorphic Morse inequalities. In the nice paper of Boucksom-Demailly-Paun-Peternell (ref. \cite{BDPP13}), the authors proposed the following conjecture on transcendental holomorphic Morse inequalities.
\begin{conjecture}
\label{conjecture}
(ref. \cite{BDPP13}) Let $X$ be an $n$-dimensional compact complex manifold.\\
$(i)$ Let $\alpha$ be a real $d$-closed $(1,1)$-form and let $X(\alpha, \leq 1)$ be the set where $\alpha$ has at most one negative eigenvalue, if $\int_{X(\alpha, \leq 1)}\alpha^{n}>0$, then the Bott-Chern class $\{\alpha\}$ contains a K\"ahler current and
$$vol( \{\alpha\})\geq  \int_{X(\alpha, \leq 1)}\alpha^{n}.$$\\
$(ii)$ Let $\{\alpha\}$ and $ \{\beta\}$ be two nef cohomology classes of type $(1, 1)$ on $X$ satisfying
$\{\alpha\}^{n}-n\{\alpha\}^{n-1}\cdot \{\beta\}>0$. Then the Bott-Chern class $\{\alpha-\beta\} $ contains a
K\"ahler current and
$$vol(\{\alpha-\beta\}) \geq \{\alpha\}^{n}-n\{\alpha\}^{n-1}\cdot \{\beta\}.$$
\end{conjecture}

In this paper, all the cohomology classes are in the Bott-Chern cohomology groups. First let us recall some definitions about the positivity of $(1,1)$-forms. Let $X$ be a compact complex manifold, and fix a hermitian metric $\omega$ on $X$.
A cohomology class $ \{\alpha\}\in H^{1,1}_{BC}(X, \mathbb{R})$ is called a nef (numerically effective) class if for any $\varepsilon>0$, there exists a smooth function $\psi_{\varepsilon}$ such that $\alpha+\varepsilon\omega+i\partial\bar \partial \psi_{\varepsilon}$ is strictly positive.
And a cohomology $(1,1)$-class $\{\alpha\}$ is called
pseudo-effective if there exists a positive current $T \in\{\alpha\}$. A positive $(1,1)$-current $T$ is called a K\"{a}hler current if $T$ is $d$-closed and $T>\delta\omega$ for some $\delta>0$. Then if $\{\alpha\}$ contains a K\"{a}hler current, we call $ \{\alpha\}$ a big class. We remark that we can also define similar positivity for $(k,k)$-classes.
For any pseudo-effective $(1,1)$-class $\{\alpha\}$, we can define its volume $$vol(\{\alpha\}):=sup_{T}\int_{X} T_{ac}^{n},$$
where $T$ ranges over all the positive currents in $\{\alpha\}$ and $T_{ac}$ is the absolutely continuous part of $T$.

\begin{remark}
Indeed, for holomorphic line bundles $L$, the above analytical definition of volume coincides with its volume in algebraic geometry, i.e., $vol(L)=\limsup_{k}\frac{n!}{k^n}h^{0}(X, kL)$.
And conjecture \ref{conjecture} holds true for holomorphic line bundles (ref. \cite{Dem85, Dem91}).
\end{remark}

In their paper \cite{BDPP13}, the authors observed that in conjecture \ref{conjecture}, $(i)$ implies $(ii)$. Thus we will call part $(ii)$ weak transcendental holomorphic Morse inequalities. Indeed, the authors proved the following theorem.

\begin{theorem}
\label{bdpp theorem}
(ref. \cite{BDPP13}) Let $X$ be a projective manifold
of dimension $n$. Then
$$vol(\omega-c_{1}(A))\geq \omega^{n}-\frac{(n + 1)^2}{4}\omega^{n-1}\cdot c_{1}(A) $$
holds for every K\"{a}hler class $\omega$ and every ample line bundle $A$ on $X$, where $c_{1}(A)$ is the first Chern class of $A$. In particular, if $ \omega^{n}-\frac{(n + 1)^2}{4}\omega^{n-1}\cdot c_{1}(A)>0$,
then $\omega-c_{1}(A)$ is big, i.e., it contains a K\"ahler current.
\end{theorem}

In this paper, we can improve the second part of theorem \ref{bdpp theorem} and get rid of the projective and rational conditions.  For part $(ii)$ of conjecture \ref{conjecture}, we get some partial results for K\"{a}hler manifolds and even for some a priori non-K\"ahler manifolds. For general compact complex manifolds, we do not know how to prove the transcendental holomorphic Morse inequalities unless we have some special metrics.

\begin{theorem}
\label{main theorem}
Let $X$ be an $n$-dimensional compact complex manifold with a hermitian metric $\omega$ satisfying $\partial\bar\partial \omega^{k}=0$ for $k=1,2,...,n-1$.
Assume $\{\alpha\}, \{\beta\}$ are two nef classes on $X$ satisfying
$$\{\alpha\}^{n}-4n\{\alpha\}^{n-1}\cdot \{\beta\}>0,$$
then $\{\alpha-\beta\}$ is a big class, i.e., there exists a K\"ahler current $T$ in $\{\alpha-\beta\}$.
\end{theorem}

Thus, our result covers the K\"ahler case and improves theorem \ref{bdpp theorem} for $n$ large enough. Moreover, the key point is that the cohomology classes $\alpha, \beta$ can be transcendental.

\begin{remark}
\label{rmk_appendix pf}
Indeed, when $n\leq 3$, we can slightly weaken the metric hypothesis. In this situation, a hermitian metric $\omega$ just satisfying $\partial\bar\partial \omega=0$ is sufficient (see the appendix).
\end{remark}

\begin{remark}
For any $n$-dimensional compact complex manifold $X$, Gauduchon's result (ref. \cite{Gau77}) tells us there always exists a metric $\omega$ such that $\partial\bar\partial \omega^{n-1}=0$.
And these metrics are called Gauduchon metrics. In particular, if $n=2$, there always exists a metric $\omega$ such that $\partial\bar\partial \omega=0$. Thus our theorem holds on any compact complex surfaces. And as a consequence, these compact complex surfaces must be K\"{a}hlerian. Indeed, this is already known thanks to the work of Buchdahl \cite{Buc99, Buc00} and Lamari \cite{Lam99a, Lam99b}.
\end{remark}

\begin{remark}
A priori, a compact complex manifold admitting a special hermitian metric described in theorem \ref{main theorem} need not be K\"{a}hlerian. However, I. Chiose in \cite{Chi13} has proved that if a compact complex manifold $X$ admits a nef class with positive top self-intersection and a hermitian metric $\omega$ with $\partial\bar \partial\omega^{k}=0$ for every $k$, then $X$ must be K\"{a}hlerian. In our proof, we do not need this fact and we will prove theorem \ref{main theorem} directly.
\end{remark}

Now let $X$ be a compact complex manifold in the Fujiki class $\mathcal{C}$, then there exists a proper modification $\mu:\widetilde{X}\rightarrow X$ such that $\widetilde{X}$ is K\"ahler. This yields the following direct corollary for compact complex manifolds in the Fujiki class $\mathcal{C}$.

\begin{corollary}
\label{coro1}
Let $X$ be a compact complex manifold in the Fujiki class $\mathcal{C}$ with dim$X=n$. Assume $\{\alpha\}$, $\{\beta\}$ are two nef classes on $X$ satisfying
$\{\alpha\}^{n}-4n\{\alpha\}^{n-1}\cdot \{\beta\}>0$,
then $\{\alpha-\beta\}$ contains a K\"ahler current.
\end{corollary}

Indeed, the proof of our theorem is inspired by I. Chiose. In section 3 of \cite{Chi13}, I. Chiose cleverly applied a lemma of Lamari (Lemma 3.3 of \cite{Lam99a}) characterizing positive currents and the ideas on mass concentration of \cite{DP04} to simplify the proof of the main theorem of Demailly and Paun. However, just as I. Chiose said, the proof of \cite{Chi13} is not independent of the proof of Demailly and Paun.
\cite{Chi13} replaced the explicit and involved construction of the
metrics $\omega_\varepsilon$ in \cite{DP04} by the abstract sequence of Gauduchon metrics given
by the Hahn-Banach theorem, via Lamari's lemma. We remark that Lamari's lemma uses the technique introduced by Sullivan in \cite{Sul76}.
We find I. Chiose's method
is useful to prove positivity of the difference of cohomology classes, at least in our case. Indeed, in addition to solving a different family of Monge-Amp\`{e}re equations, our proof almost follows the argument of \cite{Chi13}. However, our result seems not easily reachable by the mass concentration method.

\begin{remark}
\label{rmk_popovici}
Very recently, Dan Popovici \cite{Pop14} observed that keeping the same method and only changing the details
of the estimates in the Monge-Amp\`{e}re equation, one can get the optimal constant $n$ for $(1,1)$-classes on K\"ahler manifolds.
For more details of this recent improvement, we refer the readers to \cite{Pop14}.
\end{remark}

This paper is organized as follows. In section 2, we present some preliminary results. Then in section 3, we prove our main result. We give the proof of theorem \ref{main theorem}.
Finally, for the reader's convenience, we present the proof of Lamari's lemma and the proof of one additional key point in remark
\ref{rmk_appendix pf} in the appendix.\\

\noindent
\textbf{Acknowledgements}: I would like to thank Prof. Jixiang Fu for his constant encouragement and support. I would like to
thank Prof. Jean-Pierre Demailly for informing me of the recent observation of Dan Popovici. Popovici's work is an important step to go further. Thus, I would like to thank Popovici for his important observation. I would also like to thank the referee for his/her careful reading and valuable comments.

\section{Preliminaries}
Let $X$ be an $n$-dimensional compact complex manifold, for every real $(1,1)$-form $\alpha$, we have the space $PSH(X,\alpha)$ consisting of all $\alpha$-PSH functions. A function $u$ is called $\alpha$-PSH ($\alpha$-plurisubharmonic) if $u$ is an upper semi-continuous and locally integrable function such that $\alpha+i\partial\bar\partial u\geq 0$ in the sense of currents.
We have the following uniform $L^1$ bound for $\alpha$-PSH functions.

\begin{lemma}
\label{L1}
Let $X$ be an $n$-dimensional compact complex manifold with a hermitian metric $\omega$ and let $\alpha$ be a real $(1,1)$-form, then there exists a positive constant $c$ such that
$||u ||_{L^1 (\omega^n)}=\int_{X}|u|\omega^n\leq c$
for any $u\in PSH(X,\alpha)$ with $sup_{X}u=0.$
\end{lemma}

\begin{proof}
Since $X$ is compact and $\alpha$ is smooth, there exists a constant $B$ such that $B\omega>\alpha$, then $B\omega+i\partial\bar\partial u\geq 0$ for $u\in PSH(X,\alpha)$. Then the above result follows from Proposition 2.1 of \cite{DK09}.
\end{proof}

\begin{remark}
\label{L1 remark}
We will apply lemma \ref{L1} to nef classes. Let $\{\alpha\}$ be a nef class, then for any $\varepsilon>0$, there exists a smooth function $\psi_\varepsilon$ such that $\psi_\varepsilon$ is a $(\alpha+\varepsilon \omega)$-PSH function. We can always assume $sup_X \psi_\varepsilon=0$, then these $\psi_\varepsilon$ are uniformly $L^1$ bounded. This uniform $L^1$ bound is needed when we deal with the situation when $X$ only admits a metric $\omega$ with $\partial\bar \partial \omega =0$ and dim$X\leq 3$ (see the appendix).
\end{remark}

In order to apply the method of \cite{DP04} on a general compact complex manifold which maybe a priori non-K\"{a}hler, we need Tosatti's and Weinkove's result \cite{TW10} on the solvability of complex Monge-Amp\`{e}re equation on hermitian manifolds.

\begin{lemma}
\label{TW}
Let $X$ be an $n$-dimensional compact complex manifold with a hermitian metric $\omega$. Then for any smooth real-valued function $F$ on $X$, there exist a unique
real number $C > 0$ and a unique smooth real-valued function $\varphi$ on $X$ solving
$$ (\omega+i\partial\bar\partial \varphi)^n=Ce^{F}\omega^n,$$
where $\omega+i\partial\bar\partial \varphi>0$ and $sup_{X}\varphi=0$.
\end{lemma}

Finally, we state Lamari's lemma (Lemma 3.3 of \cite{Lam99a}) on the characterization of positive currents. Lamari's result is only stated for positive $(1,1)$-currents, and it also can be stated for positive $(k,k)$-currents for any $k$. However, the proof for general $k$ is the same as in Lemma 3.3 of \cite{Lam99a}. And for the reader's convenience, we will give Lamari's proof in the appendix.

\begin{lemma}
\label{lamari}
Let $X$ be an $n$-dimensional compact complex manifold and let $\Phi$ be a real $(k,k)$-form, then there exists a real $(k-1,k-1)$-current $\Psi$ such that $\Phi+i\partial\bar\partial \Psi$ is positive if and only if for any strictly positive $\partial\bar\partial$-closed $(n-k,n-k)$-form $\Upsilon$, we have $\int_{X}\Phi\wedge \Upsilon\geq 0$.
\end{lemma}

\section{The main result}
Now we can prove our main result (theorem \ref{main theorem}). Though the a priori non-K\"ahler manifolds satisfying the conditions in our theorem are actually K\"ahler, we still hope Tosatti's and Weinkove's hermitian version of Calabi-Yau theorem could apply to general
compact complex manifolds (with some new ideas). Therefore, we give the proof for the special
possibly non-K\"ahler metrics described in the statement of theorem 1.2.

\begin{proof}
(of theorem \ref{main theorem})
Firstly, fix a special hermitian metric $\omega$ satisfying $\partial\bar\partial \omega^{k}=0$ for $k=1,2,...,n-1$.
Since $\{\alpha\}, \{\beta\}$ are nef classes, for any $\varepsilon>0$, there exist smooth functions $\varphi_{\varepsilon}, \psi_{\varepsilon}$ such that $\alpha_{\varepsilon}:=\alpha+\varepsilon\omega+i\partial\bar\partial \varphi_{\varepsilon}>0$
and $\beta_{\varepsilon}:=\beta+\varepsilon\omega+i\partial\bar\partial \psi_{\varepsilon}>0$. There is no doubt we can always assume $sup\varphi_{\varepsilon}=sup\psi_{\varepsilon}=0$. And we have $\{\alpha-\beta\}=\{\alpha_{\varepsilon}-\beta_{\varepsilon}\}$, thus $\{\alpha-\beta\}$ is a big class if and only if there exists a positive constant $\delta>0$ and a $(\alpha_{\varepsilon}-\beta_{\varepsilon})$-PSH function $\theta_{\delta}$, such that
\begin{equation}
\label{big ineq1}
\alpha_{\varepsilon}-\beta_{\varepsilon}+i\partial\bar\partial \theta_{\delta}\geq \delta \alpha_{\varepsilon}.
\end{equation}
Now let us first fix $\varepsilon$. Then lemma \ref{lamari} implies (\ref{big ineq1}) is equivalent to
\begin{equation}
\label{big ineq2}
\int_{X}(\alpha_{\varepsilon}-\beta_{\varepsilon}-\delta \alpha_{\varepsilon})\wedge G\geq 0
\end{equation}
for any strictly positive $\partial\bar\partial$-closed $(n-1,n-1)$-form $G$. Then $G$ is $(n-1)$-th power of a Gauduchon metric. Now, (\ref{big ineq2}) is equivalent to
\begin{equation}
\label{big ineq3}
\int_{X}(1-\delta)\alpha_{\varepsilon}\wedge G\geq \int_{X}\beta_{\varepsilon}\wedge G.
\end{equation}
Thus, the class $\{\alpha-\beta\}=\{\alpha_{\varepsilon}-\beta_{\varepsilon}\}$ is not big is equivalent to for any $\delta_{m}\searrow 0$, there exists a Gauduchon metric $G_{m,\varepsilon}$ such that
\begin{equation}
\label{big ineq4}
\int_{X}(1-\delta_m)\alpha_{\varepsilon}\wedge G_{m,\varepsilon}< \int_{X}\beta_{\varepsilon}\wedge G_{m,\varepsilon}.
\end{equation}
Without loss of generality, we can assume $\int_{X}\beta_{\varepsilon}\wedge G_{m,\varepsilon}=1$.

By the Calabi-Yau theorem on hermitian manifold of lemma \ref{TW}, we can solve the following family of Monge-Amp\`{e}re equations
\begin{equation}
\label{cy}
\widetilde{\alpha_{\varepsilon}}^n=(\alpha_{\varepsilon}+i\partial\bar\partial u_{\varepsilon})^n=c_{\varepsilon}\beta_{\varepsilon}\wedge G_{m,\varepsilon}
\end{equation}
with $\widetilde{\alpha_{\varepsilon}}=\alpha_{\varepsilon}+i\partial\bar\partial u_{\varepsilon}, sup_{X}(\varphi_{\varepsilon}+u_{\varepsilon})=0$ and $c_{\varepsilon}=\int_{X}(\alpha_{\varepsilon}+i\partial\bar\partial u_{\varepsilon})^n$. Then $\partial\bar\partial \omega^{k}=0$ for $k=1,2,...,n-1 $ implies
\begin{equation}
\label{c_epsilon}
c_{\varepsilon}=\int_{X}(\alpha+{\varepsilon}\omega)^n\searrow c_{0}=\int_{X}\alpha^n>0.
\end{equation}
We define $M_{\varepsilon}=\int_{X}(\alpha_{\varepsilon}+i\partial\bar\partial u_{\varepsilon})^{n-1}\wedge \beta_{\varepsilon}$, then $\partial\bar\partial \omega^{k}=0$ for $k=1,2,...,n-1 $ also implies
\begin{equation}
\label{m_epsilon} M_{\varepsilon}=\int_{X}(\alpha+{\varepsilon}\omega)^{n-1}\wedge (\beta+{\varepsilon}\omega)\searrow M_{0}=\int_{X}\alpha^{n-1}\wedge \beta.
\end{equation}
We define $E_{\gamma}:=\{x\in X| \frac{\widetilde{\alpha_{\varepsilon}}^{n-1}\wedge\beta_{\varepsilon}}{G_{m,\varepsilon}\wedge\beta_{\varepsilon}}(x)>\gamma M_{\varepsilon}\}$ for some $\gamma>1$. The condition $\gamma>1$ implies $E_{\gamma}$ is a proper open subset in $X$, since we have assumed $\int_{X}\beta_{\varepsilon}\wedge G_{m,\varepsilon}=1$ and
\begin{equation}
\label{e integral}
\int_{E_{\gamma}}G_{m,\varepsilon}\wedge \beta_{\varepsilon}=\int_{E_{\gamma}}\frac{G_{m,\varepsilon}\wedge\beta_{\varepsilon}}{\widetilde{\alpha_{\varepsilon}}^{n-1}\wedge\beta_{\varepsilon}}\cdot \widetilde{\alpha_{\varepsilon}}^{n-1}\wedge\beta_{\varepsilon}<\frac{1}{\gamma M_{\varepsilon}}M_{\varepsilon}=\frac{1}{\gamma}<1.
\end{equation}
On the closed subset $X\backslash E_{\gamma}$, the definition of $E_{\gamma}$ tells us that
\begin{equation}
\label{xe ineq}
\widetilde{\alpha_{\varepsilon}}^{n-1}\wedge\beta_{\varepsilon}\leq \gamma M_{\varepsilon}\cdot G_{m,\varepsilon}\wedge \beta_{\varepsilon}.
\end{equation}
For any fixed point $p\in X\backslash E_\gamma$, choose holomorphic coordinates such that $\beta_{\varepsilon}(p)=\sum \sqrt{-1}dz_{i}\wedge d \bar z_{i}$, $\widetilde{\alpha}_{\varepsilon}(p)=\sum \sqrt{-1}\lambda_{i}dz_{i}\wedge d \bar z_{i}$ with $\lambda_1\leq\lambda_2\leq...\leq \lambda_n$.
Then at the point $p$, if we denote
$dV(p):=(\sqrt{-1})^n dz_{1}\wedge d \bar z_{1}\wedge...\wedge dz_{n}\wedge d \bar z_{n}$, then (\ref{cy}) is just
\begin{equation}
\label{cy1}
n!\lambda_1\cdot\lambda_2\cdot...\cdot \lambda_n dV(p)=c_{\varepsilon}\beta_{\varepsilon}\wedge G_{m,\varepsilon},
\end{equation}
and (\ref{xe ineq}) is
\begin{equation}
\label{xe ineq1}
(n-1)!\sum\lambda_{i_1}\cdot\lambda_{i_2}\cdot... \cdot\lambda_{i_{n-1}} dV(p)\leq \gamma M_{\varepsilon}\cdot G_{m,\varepsilon}\wedge \beta_{\varepsilon}.
\end{equation}
The above two inequalities (\ref{cy1}), (\ref{xe ineq1}) yield
$$\lambda_1 (p)\geq \frac{c_{\varepsilon}}{n\gamma M_{\varepsilon}}. $$
Since $p\in X\backslash E_\gamma$ is arbitrary, we get
\begin{equation}
\label{xe ineq form}
\widetilde{\alpha_{\varepsilon}}\geq  \frac{c_{\varepsilon}}{n\gamma M_{\varepsilon}}\cdot \beta_{\varepsilon}
\end{equation}
on $X\backslash E_\gamma$.
Now let us estimate the integral $\int_{X}\widetilde{\alpha_{\varepsilon}}\wedge G_{m,\varepsilon} =\int_{X}(\alpha+\varepsilon\omega)\wedge G_{m,\varepsilon}$.
The inequality
(\ref{xe ineq form}) implies
\begin{align}
\label{alpha ineq}
\int_{X}\widetilde{\alpha_{\varepsilon}}\wedge G_{m,\varepsilon} & \geq \int_{X\backslash E_\gamma}\widetilde{\alpha_{\varepsilon}}\wedge G_{m,\varepsilon}\\
&\geq \int_{X\backslash E_\gamma}\frac{c_{\varepsilon}}{n\gamma M_{\varepsilon}}\cdot \beta_{\varepsilon}\wedge G_{m,\varepsilon}\\
&= \frac{c_{\varepsilon}}{n\gamma M_{\varepsilon}}(\int_{X}\beta_{\varepsilon}\wedge G_{m,\varepsilon}-
\int_{E_\gamma}\beta_{\varepsilon}\wedge G_{m,\varepsilon})\\
&> \frac{c_{\varepsilon}}{n\gamma M_{\varepsilon}} (1-\frac{1}{\gamma}).
\end{align}
Take $\gamma=2$, we get
\begin{equation}
\label{ineq contr1e1}
c_{\varepsilon}-4nM_{\varepsilon}\int_{X}\widetilde{\alpha_{\varepsilon}}\wedge G_{m,\varepsilon}=c_{\varepsilon}-4nM_{\varepsilon}\int_{X}(\alpha+{\varepsilon}\omega)\wedge G_{m,\varepsilon} <0.
\end{equation}
On the other hand, (\ref{big ineq4}) implies
\begin{equation}
\label{ineq contr1e2}
\int_{X}\alpha_{\varepsilon}\wedge G_{m,\varepsilon}=\int_{X}(\alpha+{\varepsilon}\omega)\wedge G_{m,\varepsilon}< \frac{1}{1-\delta_m}.
\end{equation}
Fix a small $\varepsilon$ to be determined. Since $\int_{X}\beta_{\varepsilon}\wedge G_{m,\varepsilon}=1$,  by compactness of the sequence $\{G_{m,\varepsilon}\}$, there exists a weakly convergent subsequence which we also denote by $\{G_{m,\varepsilon}\}$ with
$$\lim_{m\rightarrow \infty} G_{m,\varepsilon}=G_{\infty,\varepsilon}$$
where the convergence is in the weak topology of currents and $G_{\infty,\varepsilon}$ is a $\partial\bar\partial$-closed positive $(n-1,n-1)$-current with
\begin{equation}
\label{leq 1}
0\leq\int_{X}(\alpha+\varepsilon\omega)\wedge G_{\infty,\varepsilon}\leq 1.
\end{equation}
Now our assumption $$\{\alpha\}^{n}-4n\{\alpha\}^{n-1}\cdot \{\beta\}>0$$ implies $$c_0 -4nM_0>0.$$ Then after taking the limit of $m$ in (\ref{ineq contr1e1}) and (\ref{ineq contr1e2}), (\ref{leq 1}) implies
 $$c_\varepsilon -4nM_\varepsilon\leq c_\varepsilon -4nM_\varepsilon \int_{X}(\alpha+\varepsilon\omega)\wedge G_{\infty,\varepsilon}<0. $$
It is clear that the contradiction is obtained in the limit when we let $\varepsilon$ go to zero.

Thus the assumption that $\{ \alpha-\beta\}$ is not a big class is not true. In other words, $\{\alpha\}^{n}-4n\{\alpha\}^{n-1}\cdot \{\beta\}>0$ implies there exists a K\"ahler current in the class $\{ \alpha-\beta\}$.
\end{proof}

After proving theorem \ref{main theorem}, corollary \ref{coro1} follows easily.

\begin{proof}
(of corollary \ref{coro1})
Since $X$ is in the Fujiki class $\mathcal{C}$, there exists a proper modification $\mu:\widetilde{X}\rightarrow X$ such that $\widetilde{X}$ is K\"{a}hler. Pull back $\alpha,\beta$ to $\widetilde{X}$, the class $\mu^* \alpha, \mu^* \beta$ are still nef classes on $\widetilde{X}$ and $\{\mu^* \alpha\}^n -4n \{\mu^* \alpha\}^{n-1}\cdot\{\mu^* \beta\}>0$.
Theorem \ref{main theorem} yields there exists a K\"{a}hler current $$\widetilde{T}\in \{\mu^* (\alpha-\beta)\}.$$
 Then $T:=\mu_* \widetilde{T}$ is our desired K\"ahler current in the class $\{\alpha-\beta\}$.
\end{proof}

\begin{remark}
We point out that, for the Bott-Chern cohomology classes $\{ \alpha^{k}-\beta^{k} \}$ on K\"ahler manifolds, we can prove a result analogous to theorem \ref{main theorem}. Its proof is almost a copy and paste of that in the $(1,1)$-case.
\emph{Let $\{ \alpha\}$ and $\{ \beta\}$ be two nef cohomology classes of type $(1, 1)$ on an $n$-dimensional compact K\"ahler manifold $X$ satisfying
the inequality $\{\alpha\}^{n}-4C_{n}^{k}\{\alpha\}^{n-k}\cdot\{\beta\}^k>0$ with $C_{n}^{k}=\frac{n!}{k!(n-k)!}$, then $\{ \alpha^{k}-\beta^{k} \}$ contains a
``strictly positive" $(k,k)$-current}. And one can get the constant $\frac{n!}{k!(n-k)!}$ by using Popovici's observation. Here, we also call such a $(k,k)$-cohomology class big and a $(k,k)$-current $T$ is called ``strictly positive" if there exist a positive constant $\delta$ and a hermitian metric $\omega$ such that $T\geq \delta\omega^k$. Fix a K\"ahler metric $\omega$, since $\{ \alpha\}$ and $\{ \beta\}$ are nef, for any $\varepsilon>0$ there exist functions $\varphi_\varepsilon, \psi_\varepsilon$ such that $\alpha_\varepsilon:=\alpha+\varepsilon \omega+i\partial\bar \partial\varphi_\varepsilon$ and $\beta_\varepsilon:=\beta+\varepsilon \omega+i\partial\bar \partial \psi_\varepsilon$ are K\"ahler metrics.
In general, unlike the $(1,1)$-case, we should note that
$\{ \alpha^k_\varepsilon -\beta^k_\varepsilon\}\neq \{ \alpha^{k}-\beta^{k} \}$. Thus the bigness of
$ \{ \alpha^k_\varepsilon -\beta^k_\varepsilon\}$ does not imply the bigness of $\{ \alpha^{k}-\beta^{k} \}$.
However, we can still apply the ideas of the proof of theorem
\ref{main theorem} by the following observation. It is obvious that $\{\alpha\}^{n}
-4C_{n}^{k}\{\alpha\}^{n-k}\cdot\{\beta_\varepsilon\}^k>0$
for $\varepsilon$ small enough. We fix such a $\varepsilon_0$, then we claim that the bigness of
$ \{\alpha^k -{\beta^k_{\varepsilon_0}}\}$ implies the bigness of $\{ \alpha^{k}-\beta^{k} \}$.
The bigness of
$ \{\alpha^k -{\beta^k_{\varepsilon_0}}\}$ yields the existence of some current $\theta_{\varepsilon_0}$ and some positive constant $\delta_{\varepsilon_0}$ such that
$$\alpha^k -\beta^k_{\varepsilon_0} +i\partial\bar\partial \theta_{\varepsilon_0}\geq \delta_{\varepsilon_0} \omega^k.$$
Then we have $\alpha^k -\beta^k +i\partial\bar\partial \widetilde{\theta}_{\varepsilon_0}\geq \delta_{\varepsilon_0} \omega^k + \gamma_{\varepsilon_0}$, where $i\partial\bar\partial \widetilde{\theta}_{\varepsilon_0}=i\partial\bar\partial \theta_{\varepsilon_0}
-\sum_{l=1}^{k}C_k ^l \sum_{p=1}^{l}C_l ^p
(i\partial\bar\partial \psi_{\varepsilon_0})^p \wedge (\varepsilon_0 \omega)^{l-p}\wedge \beta^{k-l}$
and $\gamma_{\varepsilon_0}
=\sum_{l=1}^k C_k ^l (\varepsilon_0 \omega)^l \wedge \beta^{k-l}$.
Since $\{\beta\}$ is nef, it is clear that the class $\{\gamma_{\varepsilon_0}\}$ contains a positive current $\Upsilon_{\varepsilon_0}:=\gamma_{\varepsilon_0}+i\partial\bar\partial \Phi_{\varepsilon_0}$. Then $\alpha^k -\beta^k +i\partial\bar\partial (\widetilde{\theta}_{\varepsilon_0}+\Phi_{\varepsilon_0})$ is a ``strictly positive" $(k,k)$-current in $\{\alpha^k -\beta^k\}$.
Thus we can assume $\beta$ is a K\"ahler metric in the beginning. With this assumption, we only need to show that the class $\{\alpha^k -\beta^k\}$ contains a $(k,k)$-current $T:=\alpha^k -\beta^k+i\partial\bar\partial\theta$ such that $T \geq \delta \beta^k$ for some positive constant $\delta$. This can be done as in the proof of theorem 1.2
\end{remark}

\section{Appendix}
\subsection{Lamari's lemma}
In this section, for the reader's convenience, we include the proof of lemma \ref{lamari} due to Lamari (see Lemma 3.3 of \cite{Lam99a}). The proof is an application of Hahn-Banach theorem.

\begin{lemma}
\label{lamari4}
Let $X$ be an $n$-dimensional compact complex manifold and let $\Phi$ be a real $(k,k)$-form, then there exists a real $(k-1,k-1)$-current $\Psi$ such that $\Phi+i\partial\bar\partial \Psi$ is positive if and only if for any strictly positive $\partial\bar\partial$-closed $(n-k,n-k)$-forms $\Upsilon$, we have $\int_{X}\Phi\wedge \Upsilon\geq 0$.
\end{lemma}

\begin{proof}
It is obvious that if there exists a $(k-1,k-1)$-current $\Psi$ such that $\Phi+i\partial\bar\partial \Psi$ is positive, then for any strictly positive $\partial\bar\partial$-closed $(n-k,n-k)$-form $\Upsilon$, we have $\int_{X}\Phi\wedge \Upsilon\geq 0$.

In the other direction, assume $\int_{X}\Phi\wedge \Upsilon\geq 0$ for any strictly positive $\partial\bar\partial$-closed $(n-k,n-k)$-form $\Upsilon$. Firstly, let us define some subspaces in the real vector space $\mathcal{D}^{n-k,n-k}_{\mathbb{R}}$ consisting of real smooth $(n-k,n-k)$-forms with Fr\'{e}chet topology. We denote
\begin{align*}
&E=\{\Upsilon\in \mathcal{D}^{n-k,n-k}_{\mathbb{R}}| \partial\bar \partial\Upsilon =0\},\\
&C_1 =\{\Upsilon\in E| \Upsilon \ \text{is strictly positive}\},\\
&C_2 =\{\Upsilon\in \mathcal{D}^{n-k,n-k}_{\mathbb{R}}| \Upsilon \ \text{is strictly positive}\}.
\end{align*}
Then if we consider $\Phi$ as a linear functional on $\mathcal{D}^{n-k,n-k}_{\mathbb{R}}$, we have $\Phi|_{C_{1}}\geq 0$.

If there exists a $\Upsilon_{0}\in C_1$ such that $\Phi(\Upsilon_{0})=0$. Then we consider the affine function $f(t)=\Phi(t\alpha+(1-t)\Upsilon_{0})$, where $\alpha\in E$ is fixed. The function $f(t)$ satisfies $f(0)=0$, moreover, since $\Upsilon_{0}\in C_1$ is strictly positive and $X$ is compact,
 for $\varepsilon$ small enough, $f(\pm \varepsilon)\geq 0$ by the assumption. This implies $f(t)\equiv 0$, in particular, $f(1)=\Phi(\alpha)=0$. By the arbitrariness of $\alpha\in E$, we get $\Phi|_E =0$, thus $\Phi=i\partial\bar\partial \Psi$ for some current $\Psi$. So in this case, we have $\Phi+i\partial\bar \partial(-\Psi)=0$.

Otherwise, for any $\Upsilon_{0}\in C_1$, we have $\Phi(\Upsilon_{0})>0$, i.e., $\Phi|_{C_1}>0$. Since $\Phi$ can be seen as a linear functional on $\mathcal{D}^{n-k,n-k}_{\mathbb{R}}$, we can define its kernel space $ker \Phi$, it's a linear subspace. We denote $F=E\cap ker \Phi $, then $F\cap C_2 = \emptyset$. Next, we need the following geometric Hahn-Banach theorem or Mazur's theorem.

$\bullet$ \emph{Let $M$ be a vector subspace of the topological vector space $V$. Suppose $K$ is a non-empty convex open subset of $V$ with
$K\cap M= \emptyset$. Then there is a closed hyperplane $N$ in $V$ containing $M$ with $K\cap N= \emptyset$.}

The above theorem yields there exists a real $(k,k)$-current $T$ such that $T|_{F} =0$ and $T|_{C_2}>0$. Take $\Upsilon\in C_1$, then $\Phi(\Upsilon), T(\Upsilon)$ are both positive. So there exists a positive constant $\lambda$ such that $(\Phi-\lambda T)(\Upsilon)=0$. Observe that $F$ is codimension one in $E$ and $\Upsilon\in E\backslash F$, thus $\Phi-\lambda T$ is identically zero on $E$. This fact yields there exists a current $\Psi$ such that $\Phi+i\partial\bar\partial \Psi=\lambda T\geq 0$.
\end{proof}

\subsection{Proof of Remark \ref{rmk_appendix pf}}
\begin{proof}
From the proof of theorem \ref{main theorem}, we know that a key ingredient is the dependence of $c_{\varepsilon}, M_\varepsilon$ on $\varepsilon$ as $\varepsilon$ tends to zero. These constants come from the following family of Monge-Amp\`{e}re equations:
\begin{equation*}
\label{cy3}
\widetilde{\alpha_{\varepsilon}}^n=(\alpha_{\varepsilon}+i\partial\bar\partial u_{\varepsilon})^n=c_{\varepsilon}\beta_{\varepsilon}\wedge G_{m,\varepsilon}.
\end{equation*}
In this case, the uniform $L^1$ bound in lemma \ref{L1} plays an important role. For $c$ large enough, we have $\psi_\varepsilon, \varphi_\varepsilon+u_\varepsilon$ are all $c\omega$-PSH.
Since $sup\psi_\varepsilon =sup (\varphi_\varepsilon+u_\varepsilon)=0$, if we denote $\varphi_\varepsilon+u_\varepsilon$ by $\eta_\varepsilon$, we have
\begin{equation}
\label{L1 1}
 ||\psi_\varepsilon ||_{L^1(\omega^n)}+||\eta_\varepsilon||_{L^1(\omega^n)}<C
\end{equation}
for a uniform constant $C$.

Firstly, assume $n=3$, then by (\ref{L1 1}) and $\partial\bar\partial \omega=0$
\begin{align*}
c_\varepsilon &=\int_{X}(\alpha+\varepsilon\omega+i\partial\bar\partial\eta_\varepsilon)^3\\
&=\int_{X}(\alpha+i\partial\bar\partial\eta_\varepsilon)^3+\varepsilon^3 \omega^3\\
&+3\varepsilon\omega\wedge(\alpha+i\partial\bar\partial\eta_\varepsilon)^2+3\varepsilon^2 \omega^2 \wedge(\alpha+i\partial\bar\partial\eta_\varepsilon)\\
&=\int_{X}\alpha^3 +O(\varepsilon).
\end{align*}
Thus, $c_\varepsilon>0$ for $\varepsilon$ small and $\lim_{\varepsilon\rightarrow 0}c_\varepsilon= c_0$. Similarly, by the definition of $M_\varepsilon$, we have
\begin{align*}
M_\varepsilon &=\int_{X}(\alpha+\varepsilon\omega+i\partial\bar\partial\eta_\varepsilon)^2 \wedge (\beta+\varepsilon\omega+i\partial\bar\partial\psi_\varepsilon)\\
&=\int_{X}((\alpha+i\partial\bar\partial\eta_\varepsilon)^2 +\varepsilon^2 \omega^2 +2(\alpha+i\partial\bar\partial\eta_\varepsilon)\wedge \varepsilon\omega)\wedge \beta\\
&+\int_{X}(\cdots)\wedge \varepsilon\omega+\int_{X}(\cdots)\wedge i\partial\bar\partial\psi_\varepsilon\\
&=r_\varepsilon +s_\varepsilon +t_\varepsilon.
\end{align*}
By (\ref{L1 1}) and $\partial\bar\partial \omega=0$ again, it is easy to see that
\begin{align*}
&r_\varepsilon=\int_{X}\alpha^2\wedge\beta+2\varepsilon\alpha\wedge\beta\wedge \omega +O(\varepsilon^2),\\
&s_\varepsilon=\varepsilon\int_{X}\alpha^2 \wedge \omega + O(\varepsilon^2),\\
&t_\varepsilon=O(\varepsilon^2).
\end{align*}
So by the above calculation, we get $\lim_{\varepsilon\rightarrow 0} M_\varepsilon =M_0= \int_{X}\alpha^2\wedge\beta$.
\emph{A priori, it is not obvious whether we have $M_\varepsilon>0$ for $\varepsilon>0$ small enough. We claim $M_\varepsilon>0$ and this depends on $c_0=\int_X \alpha^3 >0$.} Since $\alpha$ and $\beta$ are nef, we only need to verify $\int_{X}\alpha^2 \wedge \omega>0$.
Firstly, inspired by \cite{Dem93}, we solve the following family of complex Monge-Amp\`{e}re equations
\begin{equation}
\label{ma3}
(\alpha+\varepsilon\omega+i\partial\bar \partial u_\varepsilon)^3 =U_\varepsilon \omega^3
\end{equation}
where $sup u_\varepsilon =0$ and $U_\varepsilon={\int_X (\alpha+\varepsilon\omega+i\partial\bar \partial u_\varepsilon)^3}/ {\int_X \omega^3}$ is a positive constant. By the above estimate of $c_\varepsilon$, we know
\begin{equation}
\label{U}
U_\varepsilon =\frac{\int_X (\alpha+\varepsilon\omega+i\partial\bar \partial u_\varepsilon)^3}{\int_X \omega^3}
=\frac {c_0 + O(\varepsilon)}{\int_X \omega^3}.
\end{equation}
It is easy to see that
\begin{equation}
\int_X (\alpha+\varepsilon\omega+i\partial\bar \partial u_\varepsilon)^2 \wedge \omega
=\int_X \alpha^2 \wedge \omega + O(\varepsilon),
\end{equation}
or equivalently,
\begin{equation}
\label{U1}
\int_X \alpha^2 \wedge \omega =
\int_X (\alpha+\varepsilon\omega+i\partial\bar \partial u_\varepsilon)^2 \wedge \omega- O(\varepsilon).
\end{equation}
Then the pointwise inequality
$$ \frac{(\alpha+\varepsilon\omega+i\partial\bar \partial u_\varepsilon)^2 \wedge \omega}{\omega^ 3}\geq \big(\frac{(\alpha+\varepsilon\omega+i\partial\bar \partial u_\varepsilon)^3}{\omega^ 3}\big )^{\frac{2}{3}}\cdot \big(\frac{\omega^ 3}{\omega^ 3}\big )^{\frac{1}{3}}$$
implies
\begin{equation}
\label{U2}
\int_X \alpha^2 \wedge \omega \geq
{U_\varepsilon}^{\frac{2}{3}}\int_X\omega^3 - O(\varepsilon)
=(c_0 +O(\varepsilon))^{\frac{2}{3}}(\int_X \omega^3)^{\frac{1}{3}} - O(\varepsilon).
\end{equation}
Then $c_0 >0$ yields the existence of some positive constant $c'$ such that $$\int_X \alpha^2 \wedge \omega \geq c'.$$
And this concludes our claim that $M_\varepsilon>0$ for $\varepsilon$ small enough. With these preparations, the proof of Remark \ref{rmk_appendix pf} when $n=3$ is the same as
theorem \ref{main theorem}.
Similarly, we can also prove the case
when $n<3$.
\end{proof}

\textsc{Institute of Mathematics, Fudan University, Shanghai 200433, China} \\
\textsc{Jian Xiao} \\
\verb"Email: jxiao10@fudan.edu.cn"

\end{document}